\documentclass[english]{extarticle}
\usepackage[T1]{fontenc}
\usepackage[latin9]{inputenc}
\usepackage{geometry}
\usepackage{color}
\usepackage{babel}
\usepackage{float}
\usepackage{amsmath}
\usepackage{amsthm}
\usepackage{amssymb}
\usepackage{graphicx}
\usepackage[all]{xy}
\usepackage[unicode=true,pdfusetitle,
 bookmarks=true,bookmarksnumbered=true,bookmarksopen=false,
 breaklinks=false,pdfborder={0 0 1},backref=false,colorlinks=true]
 {hyperref}
\hypersetup{
 linkcolor=blue,urlcolor=blue,citecolor=blue,pdfstartview={FitH},hyperfootnotes=false}

\makeatletter

\floatstyle{ruled}
\newfloat{algorithm}{tbp}{loa}
\providecommand{\algorithmname}{Algorithm}
\floatname{algorithm}{\protect\algorithmname}

\theoremstyle{plain}
\newtheorem{thm}{\protect\theoremname}

\usepackage{babel}
\usepackage{lettrine}
\let\myFoot\footnote
\renewcommand{\footnote}[1]{\myFoot{#1\vspace{3mm}}}
\usepackage{mlmodern}

\usepackage{tikz}

\providecommand{\theoremname}{Theorem}

\makeatother

\providecommand{\theoremname}{Theorem}

\begin{document}
\title{Brownian symmetrization of planar domains }
\author{Kais Hamza \thanks{Monash University.} , Maher Boudabra \thanks{King Fahd University for Petroleum and Minerals.}}
\maketitle
\begin{abstract}
One of aims of this note is to capture the interest of the mathematical
community to a novel transformation, which we shall call ``Brownian
symmetrization''. This transformation arises from the solution of
the planar Skorokhod embedding problem, first appeared in \cite{gross2019}.
Brownian symmetrization shares some properties with the famous Steiner
symmetrization. However, we show that these two transformations are
not the same and they do not affect each other. 
\end{abstract}

\section{Introduction}

In \cite{gross2019}, the following question has been posed and answered
: Given a distribution $\mu$ with zero mean and finite second moment,
is there a simply connected domain $U$ (containing the origin) such
that if $(Z_{t}=X_{t}+Y_{t}i)_{t\geq0}$ is a standard planar Brownian
motion, then $X_{\tau_{U}}=\Re(Z_{\tau_{U}})$ samples as $\mu$ and
$\mathbb{E}(\tau_{U})<+\infty$, where $\tau_{U}$ is the exit time
from $U$ (We shall omit the subscript of $\tau$ if the domain is
clear from the context) ? In other words, we seek a domain $U$ so
that $Z_{\tau_{U}}\sim\mu$. A domain that solves the problem shall
be called a $\mu$-domain \cite{Boudabra2020a}. The answer of the
question is affirmative. The construction process adopted by the author
generates a domain $U$ with the following geometric properties : 
\begin{itemize}
\item $U$ is symmetric over the real line. 
\item If $\mu(\{x\})>0$ then $\partial U$ contains a vertical line segment
(possibly infinite). 
\end{itemize}
One can notice that if $(a,b)$ is not supported by $\mu$, i.e $(a,b)$
is null set for $\mu$ \footnote{~It means that the c.d.f of $\mu$ is constant on $(a,b)$.},
then, due to the simple connectivity requirement, any $\mu$-domain
would contain the vertical strip $(a,b)\times(-\infty,+\infty)$.
The constraint of the finiteness of the average of the exit time $\tau$
is the one making the problem challenging. To see this, consider the
Redmacher distribution $\mu=\frac{1}{2}\delta_{-1}+\frac{1}{2}\delta_{1}$.
The real part of the stopped Brownian motion samples according to
$\mu$ for the vertical strip $\{-1\leq\Re(z)\leq1\}$ and the "ripped"
plane $\mathbb{C}-\{\Re(z)=\pm1,\vert\Im(z)\vert\geq1\}$. But, in
average, the Brownian motion spends a finite amount of time to leave
the strip while it needs an infinite amount of time to get out of
the "ripped" plane. That is, the vertical strip is the correct domain
for such a distribution. For the sake of clarity, we give an outline
of the key idea of the proof (See \cite{gross2019} for more details).
The core tool of Gross idea is the following theorem. 
\begin{thm}[Lévy's theorem]
\label{conformal }  Let $f:U\rightarrow\mathbb{C}$ be a non constant
univalent function and $Z_{t}$ be a planar Brownian motion running
inside $U$. Then there is a planar Brownian motion $W_{t}$ such
that $f(Z_{t})=W_{\sigma(t)}$ with 
\[
\sigma(t)=\int_{0}^{t}\vert f'(Z_{s})\vert^{2}ds
\]
and $t\in[0,\tau_{U}]$. 
\end{thm}

Theorem \ref{conformal } is commonly known as the conformal invariance
principle. That is, the image of a planar Brownian motion under the
action of non constant analytic function is yet another planar Brownian
motion, except it runs in a different speed. With the same notation,
it is equivalent to say that $f(Z_{\sigma^{-1}(t)})$ is planar Brownian
motion. The random time $\sigma(\tau_{U}^{-})$ is referred to as
the projection of the exit time $\tau_{U}$. It is not always the
case that $\sigma(\tau_{U}^{-})$ is the exit time of $f(U)$ unless
$f$ is proper which is the case when $f$ is univalent. In particular,
the law of $f(Z_{\tau_{U}})$ is obtainable from that of $Z_{\tau_{U}}$(we
refer the reader to \cite{markowsky2018distribution} for more illustrations
of this fact). In order to solve the problem, Gross's smart idea was
to construct a univalent map acting on the unit disc such that $\Re(f(e^{ti}))$
samples as $\mu$. By the aforementioned conformal invariance principle,
the distribution of the real part of a stopped planar Brownian motion
on the boundary of $f(\mathbb{D})$ follows exactly $\mu$. If $F(x)$
denotes the c.d.f of $\mu$ then the pseudo inverse of $\mu$ (also
referred to as the quantile function) is defined by $G_{\mu}(u):=\inf\{x\mid F(x)\geq u\}$.
A well known property of $G_{\mu}$ is that 
\[
G_{\mu}(\mathrm{Uni}(0,1))\sim\mu
\]
which follows immediately from the definition of $G_{\mu}$ \cite{devroye2006nonuniform}.
Gross then considered the Fourier series of the function $\varphi(\theta):=G_{\mu}(\frac{\vert\theta\vert}{\pi})$
over the interval $(-\pi,\pi)$, say
\[
\sum_{n=1}^{\infty}\widehat{\varphi}(n)\cos(n\theta)
\]
 where $\widehat{\varphi}(n)$ is the $n^{th}$ Fourier coefficient
( $\widehat{\varphi}(0)$ is zero because $\mu$ is centered). The
clue result of Gross is 
\begin{thm}
The map 
\[
\psi(z):=\sum_{n=1}^{\infty}\widehat{\varphi}(n)z^{n}
\]
is one to one in $\mathbb{\mathbb{D}}$. 
\end{thm}

In \cite{boudabra2019remarks}, the authors extended the construction
to any distribution with some finite $p^{th}$ moment with $p>1$.
Their extension uses a deep result in Fourier analysis: the Carleson-Hunt
theorem \cite{fefferman1973pointwise,Reyna2004}. It states that Fourier
series converges point-wisely a.e when the underlying function is
in $L^{p}$ for some $p>1$, which makes the provided solution a strong
one. As one can see, the technique to produce such $\mu$-domains
is purely analytic as it relies on the Fourier expansion of the quantile
function $\varphi(\theta)$. In particular, the geometric properties
are encoded by the conformal map $\psi(z)$. For example, the symmetry
with respect to the real axis is due to the fact that the coefficients
of $\psi(z)$ are real. In \cite{boudabra2019remarks}, the two authors
gave a criterion to guarantee the uniqueness of Gross construction.
That is, they proved that if a $\mu$-domain is symmetric over the
real line and $\Delta$-convex then it is unique, provided that $\mathbb{E}(\tau^{\frac{p}{2}})<\infty$
for some $p>1$. The $\Delta$-convexity property means that the segment
joining any point of the domain and its symmetric point w.r.t to the
real line remains inside the domain \cite{boudabra2020maximizing,boudabra2019remarks}.
We shall call such a domain of Gross type. In particular, there is
one to one correspondence between distributions and Gross domains.
For example, the unit disc corresponds to the scaled and centered
arc-sine law 
\[
d\mu(x)=\frac{1}{\pi\sqrt{1-x^{2}}}1_{\{x\in(-1,1)\}}.
\]
The area of Gross $\mu$-domain is given by 
\begin{equation}
\int_{\mathbb{D}}\vert\psi'(z)\vert^{2}d\sigma(z)=\pi\sum_{n=1}^{\infty}n\widehat{\varphi}(n)^{2}.\label{eqn:area}
\end{equation}
where $d\sigma(z)$ denotes the infinitesimal area \cite{duren2001univalent}.
The identity (\ref{eqn:area}) shows that the area of Gross domains
is at least $\pi\widehat{\varphi}(1)$, which corresponds to the centered
disc of radius $\vert\widehat{\varphi}(1)\vert$. 

The following algorithm can be coded to generate the $\mu$-domain
using Gross method. Note that the Cartesian equation of the boundary
is obtainable using Hilbert transform \cite{Boudabra2020}

\begin{algorithm}[H]
\label{alg}
\begin{enumerate}
\item Run $X=\Re(Z_{\tau})$ where $Z_{\tau}$ is the stopped Brownian motion.
\item Compute $F_{X}(x)=\mathbb{P}(X\leq x)$.
\item Compute $G(u)=F_{x}^{-1}(u)$.
\item Compute $\widehat{\varphi}(n)=\frac{2}{\pi}\int_{0}^{\pi}G(\frac{\theta}{\pi})\cos(n\theta)d\theta$.
\item Plot $\psi(e^{ti})$ for $t\in(-\pi,\pi)$.
\end{enumerate}
\caption{Algorithm for Gross solution}
\end{algorithm}

The planar curve $\psi(e^{ti})$ is the radial limit of $\psi$, i.e
$\psi(e^{ti}):=\lim_{r\rightarrow1}\psi(re^{ti})$. The existence
of the radial limit follows from the theory of Hardy spaces \cite{duren2000theory,Rudin2001}.
More generally, algorithm (\ref{alg}) applies to $U$ as long as
its underlying distribution $\mu$ has some finite $p^{th}$ moment
with $p>1$. This condition is obviously fulfilled by domains contained
in vertical strips since bounded distributions have finite $p^{th}$
moment for any $p>1$. Although, catching some $p>1$ is not an easy
task for general unbounded domains. Such a $p$ is calculable in the
case of wedges. More precisely, if $W_{2\theta}$ denotes the wedge
$\{\mathbf{Arg}(z)\in(-\theta,\theta)\}$ then $p^{th}$ moments of
$\mu$ are finite whenever $p<\frac{\pi}{2\theta}$ \cite{boudabrathesis,burkholder1977exit}.
In particular, any wedge strictly included in the right half plane
carries a distribution $\mu$ of a finite $p^{th}$ moment for some
$p>1$. In a striking work of D. Burkholder \cite{burkholder1977exit},
the author provided a very strong result connecting planar Brownian
motion and analytic functions. More precisely, he proved the following
equivalences 
\[
\xymatrix{\phantom{.} & H_{p}(f)<+\infty & \phantom{.}\ar@/^{.5pc}/[d]\\
\phantom{.}\ar@/^{.5pc}/[u] & \mathbf{\mathbf{E}}(\tau^{\frac{p}{2}})<+\infty & \phantom{.}\ar@/^{.5pc}/[d]\\
\phantom{.}\ar@/^{.5pc}/[u] & \mathbf{\mathbf{E}}(Z_{\tau}^{*p})<+\infty & \phantom{.}\ar@/^{.5pc}/[d]\\
\phantom{.}\ar@/^{.5pc}/[u]^{\mathbf{\mathbf{E}}(\ln(\tau))<+\infty} & \mathbf{\mathbf{E}}(\lvert Z_{\tau}\rvert^{p})<+\infty & \phantom{.}
}
\]
where $Z_{\tau}^{*}:=\sup_{0\leq t\leq\tau}\vert Z_{\tau}\vert$ and
$H_{p}(f)$ is the $p^{th}$ Hardy norm \footnote{The word "norm" is an abuse of terminology. Hardy norms are "true"
norms only when $p\geq1$ since we lose the convexity property if
$p<1$.} of $f:\mathbb{D}\rightarrow U$ defined by 
\[
H_{p}(f)=\sup_{0\leq r<1}\left\{ \frac{1}{2\pi}\int_{0}^{2\pi}|f(re^{\theta i})|^{p}d\theta\right\} ^{\frac{1}{p}}.
\]
Note that the requirement on $\mathbb{E}(\ln(\tau))$ has been relaxed
to the finiteness of $\mathbb{E}(\ln(\tau+1))$ which was shown to
cover all simply connected domains (See \cite{boudabrathesis} for
more details). The supremum of $p$ such that $H_{p}(f)$ is finite
is called Hardy number of $f$, first introduced in \cite{hansen1970hardy}. 

The technique developed to construct $\mu$-domains in \cite{gross2019}
triggered us to consider an underlying geometric transformation that
we shall call Brownian symmetrization (or Gross symmetrization). That
is, consider a bounded domain $U$ and run a planar Brownian motion
$(Z_{t})_{t}$ until it hits $\partial U$ and denote by $\mu$ the
distribution of $\Re(Z_{\tau})$ with $\tau$ being the exit time
of $Z_{t}$ from $U$. The Brownian symmetrization of $U$ is simply
the $\mu$-domain constructed as in \cite{gross2019}. We shall denote
the Brownian symmetrization by $\mathfrak{B}(U)$

\[
\xymatrix{W\ar[rr]^{\text{extract \ensuremath{\mu}}} &  & \mu\ar[rr] &  & W=\mathfrak{B}(U)=\text{Gross domain generated form \ensuremath{\mu}}}
\]

\section{Steiner symmetrization and Brownian symmetrization are different}

Steiner symmetrization is a geometric transformation technique used
in convex geometry and differential geometry. It involves replacing
a given geometric shape, typically a bounded set in Euclidean space,
with its symmetric counterpart with respect to a chosen axis of symmetry.
More precisely, Steiner symmetrization works as follows : Let $U$
be a planar bounded domain containing the origin. For every $z$ in
$A$ draw the vertical line $x=\Re(z)$, say $L_{x}$, and let $\ell_{x}$
be the length of $L_{x}\cap U$. The Steiner symmetrization of $U$,
denoted here by $\mathfrak{S}(A)$, is simply the union of all segments
$\{\Re(z)+t\ell_{\Re(z)}\mid-\frac{1}{2}<t<\frac{1}{2}\}_{z\in U}$.
A well known property of Steiner symmetrization is that it preserves
the area of $U$ and does not increase the perimeter \cite{chlebik2005perimeter}.
Furthermore, one can remark that $\mathfrak{S}(U)$ is $\Delta$-convex
and symmetric over the real line. In particular, both Brownian symmetrization
and Steiner symmetrization share two common properties : $\Delta$-convexity
and the symmetry over the real axis. Using the uniqueness criterion
in \cite{Boudabra2020}, the following identities hold
\[
\mathfrak{S}\circ\mathfrak{B}(U)=\mathfrak{B}(U)
\]
and 
\[
\mathfrak{B}\circ\mathfrak{S}(U)=\mathfrak{S}(U)
\]

In \cite{gross2019}, the author provided the following example to
show the non uniqueness of the solution of the planar Skorokhod embedding
problem. Let $\mu$ be the probability distribution of $\Re(Z_{\tau})$
where $\tau$ is the exit time from the domain $B$ limited by the
blue boundary. The Brownian symmetrization of $B$ is the domain $R$
limited by the red boundary. Note that the domain $R$ is obtained
empirically. 
\begin{center}
\begin{figure}[H]
\begin{centering}
\includegraphics[width=5cm,height=5cm,keepaspectratio]{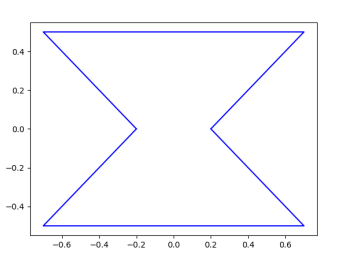}\includegraphics[width=5cm,height=5cm,keepaspectratio]{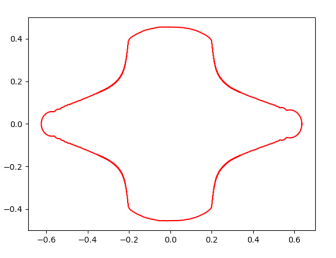} 
\par\end{centering}
\caption{The domain $R$ on the right is the Brownian symmetrization of the
domain $B$ on the left. }
\end{figure}
\par\end{center}

It is easy to see that the Steiner symmetrization of $B$ is the domain
$G$ limited by the green boundary.

\begin{figure}[H]
\begin{centering}
\includegraphics[width=4.5cm,height=4.5cm,keepaspectratio]{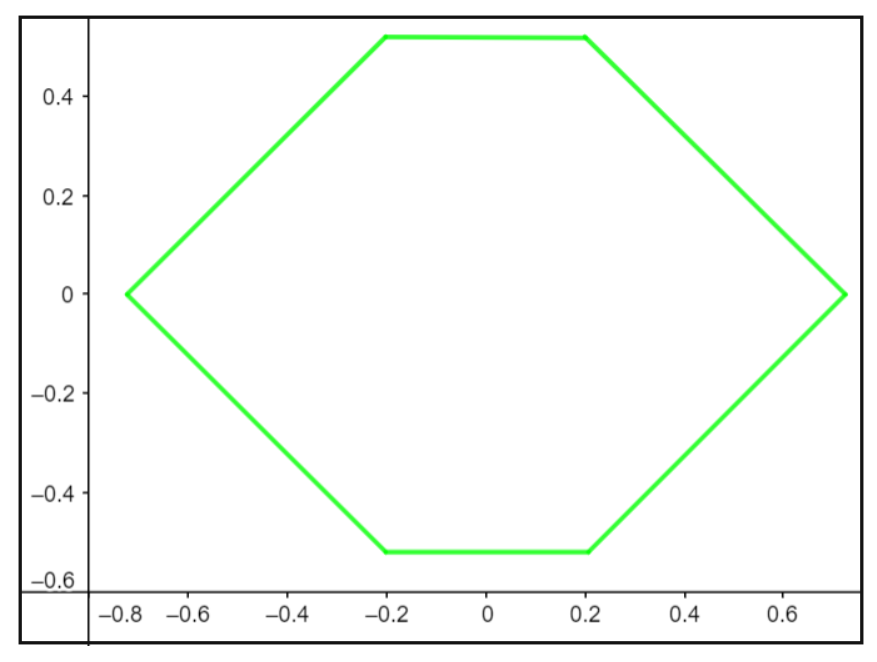} 
\par\end{centering}
\caption{The Steiner symmetrization of the domain $B$.}
\end{figure}

If we believe the simulation then the two domains are not the same,
i.e 
\[
\mathfrak{B}(B)=R\neq\mathfrak{S}(B)=G.
\]
Although $R$ and $B$ are far of being the same, we provide a convincing
argument that proves that the two symmetrizations are different. 
\begin{thm}
\label{T} The Steiner symmetrization and the Brownian symmetrization
are different. 
\end{thm}

\begin{proof}
Let $U$ be the following domain 
\[
U=\{-1<x\leq0,-\frac{1}{2}<y<1\}\cup\{0<x\leq1,\frac{1}{2}<y<-1\}\cup\{x=0,\vert y\vert<\frac{1}{2}\}.
\]

\begin{figure}[H]
\begin{centering}
\includegraphics[width=7cm,height=7cm,keepaspectratio]{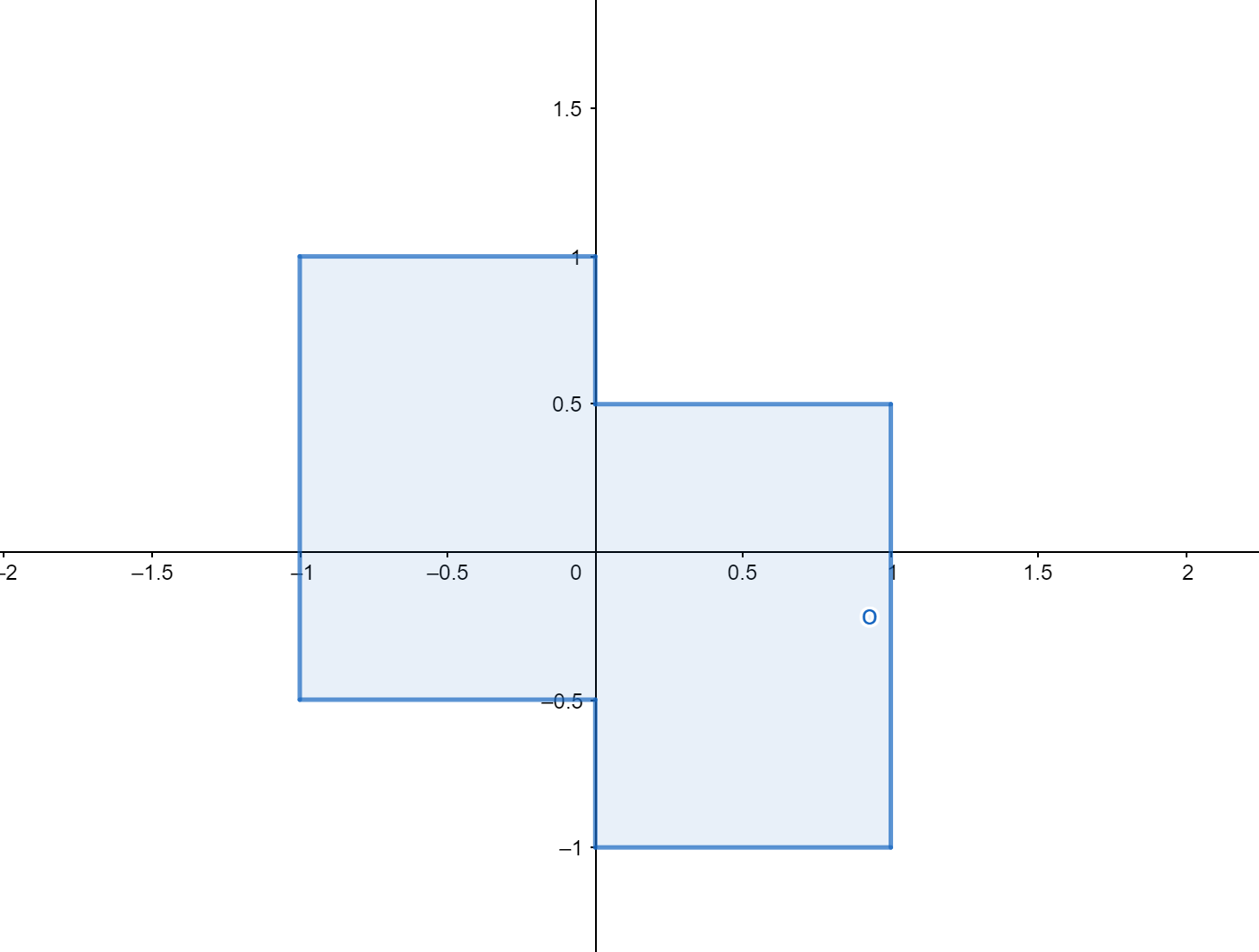} 
\par\end{centering}
\caption{The domain $U$.}
\end{figure}

The Steiner symmetrization of $U$ is the the rectangle $S=\{\vert x\vert<0,\vert y\vert<\frac{3}{4}\}$.

\begin{figure}[H]
\begin{centering}
\includegraphics[width=6cm,height=6cm,keepaspectratio]{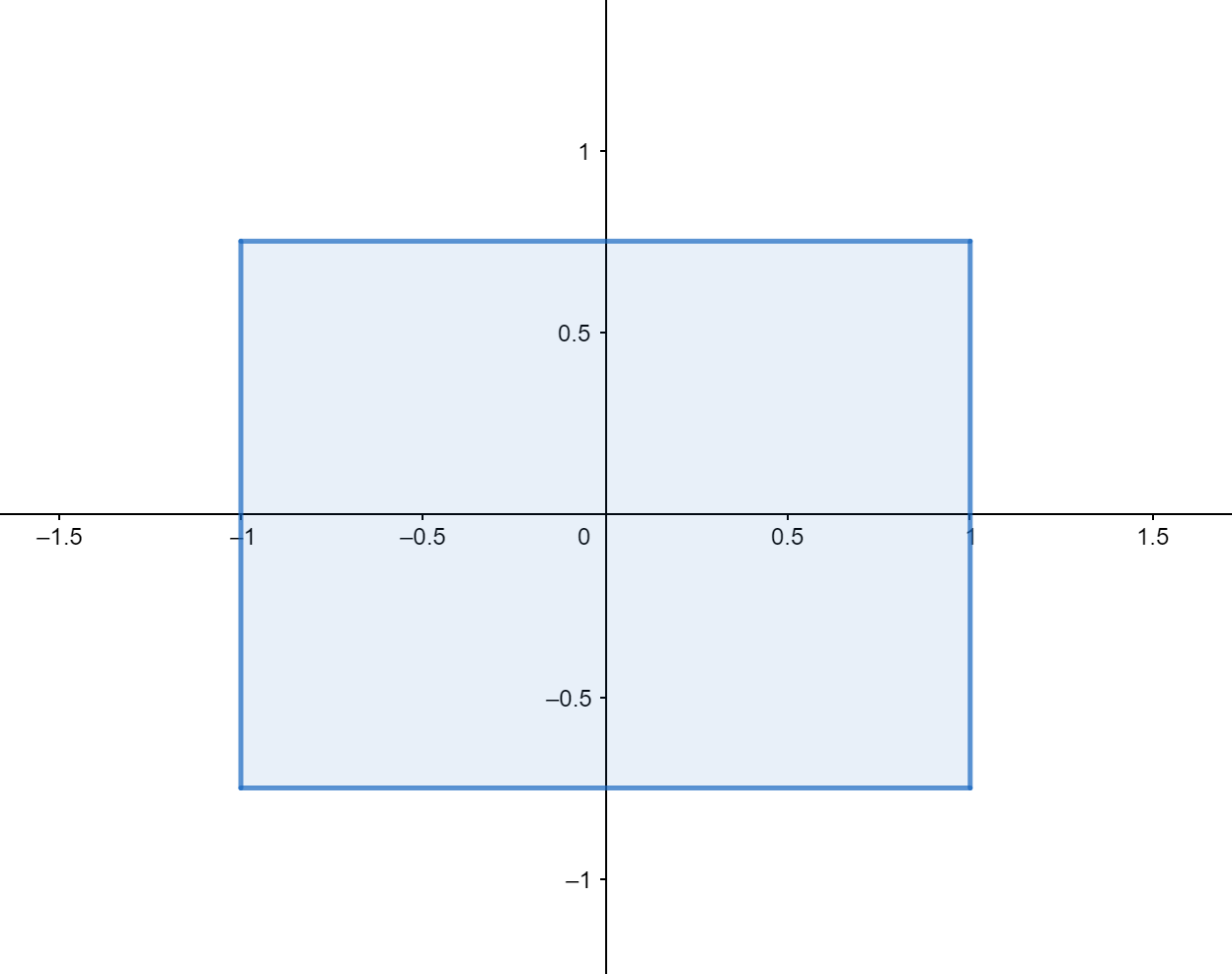} 
\par\end{centering}
\caption{The Steiner symmetrization of $U$.}
\end{figure}

Let $\mu$ be the probability law of $\Re(Z_{\tau_{U}})$. The distribution
$\mu$ has three point masses namely at $x=-1,x=0$ and $x=1$. This
fact uses F. and M. Riesz Theorem that asserts that the harmonic measure
of a simply connected domain with rectifiable boundary is mutually
continuous with respect to the arc length. However, we provide a short
explanation articulating on the conformal invariance of the Brownian
motion \cite{morters2010brownian}. Let $f$ be a conformal map from
the unit disc onto $U$ that sends zero to zero. As the boundary of
$U$ is of Jordan type (in addition of being rectifiable) then by
Carathéodory's theorem \cite{caratheodory1913gegenseitige} the function
$f$ extends continuously to the closed unit disc without losing the
injectivity property. For any portion $E$ of the boundary the harmonic
measure of $E$ is simply $\frac{\vert f^{-1}(E)\vert}{2\pi}$ where
$\vert\cdot\vert$ denotes the Lebesgue measure on the unit circle.
This follows from the fact that the stopped Brownian motion on the
unit circle is uniformly distributed. That is, $E$ has zero harmonic
measure if and only its arc length is zero ($f$ is an homeomorphism).
As the vertical parts of the Boundary of $U$ are of positive lengths
then they induce point masses for $\mu$. Therefore the Brownian symmetrization,
say $V$, would preserve these point masses. According to the construction
process, the boundary of $V$ will contain three vertical segments
(and their symmetric images over the real line) at $x=-1,x=0$ and
$x=1$. Therefore $V=\mathfrak{B}(U)\neq S=\mathfrak{S}(U)$. 
\end{proof}

\section{Concluding remarks}

The Skorokhod embedding problem has another different solution given
in \cite{Boudabra2020a}, which can be studied as another version
of symmetrization. Unlike Gross's method, it generates unbounded domains
even for bounded distributions. More precisely, it creates $\mu$-domains
that are $\Delta^{\infty}$ convex, i.e with the property that the
upward half ray starting at any point of the domain remains inside
it (More details are in \cite{Boudabra2020a}). In addition, domains
with such a property are proved to be unique under some technical
requirement on their exit times. In \cite{mariano2020}, an interesting
optimization problem was suggested : Fix a probability distribution
$\mu$, among all $\mu$-domains, which one has the highest and lowest
rate, where the rate is half of the principal Dirichlet eigenvalue.
The authors answered the question regarding the lowest rate for $\mu$
being the uniform distribution on $(-1,1)$. It turns out that their
solution is a particular case of the solution appeared in \cite{Boudabra2020a}.
The problem of the highest rate is still open. Although, it is guessed
that the Brownian symmetrization maximizes the rate. In the same context,
it is known that Steiner symmetrization minimizes the principal Dirichlet
eigenvalue. So a natural question presents itself : What would be
the effect of the Brownian symmetrization on such an eigenvalue? Another
related question is what is the effect of the Brownian symmetrization
on the perimeter? 

At the end, we provide this example followed by some conjectures.
Let $U$ be the set $\{x^{2}+(y+\kappa)^{2}<1\}$ where $\kappa\in(0,1)$.
The law of $\Re(Z_{\tau})$ is given by 
\[
d\mu(x)={\textstyle \frac{(1-\kappa^{4})}{\pi\left((1-\kappa^{2})^{2}+4\kappa^{2}x^{2}\right)\sqrt{1-x^{2}}}}
\]
\cite{boudabra2019remarks}. The c.d.f of $\mu$ is 
\[
\begin{alignedat}{1}F_{\mu}(x) & =\frac{1}{\pi}\arctan\left({\textstyle \frac{1+\kappa^{2}}{1-\kappa^{2}}\frac{x}{\sqrt{1-x^{2}}}}\right)+{\textstyle \frac{1}{2}}\\
 & =\frac{1}{\pi}\arctan\left({\textstyle \eta\frac{x}{\sqrt{1-x^{2}}}}\right)+{\textstyle \frac{1}{2}}
\end{alignedat}
\]
and so
\[
\varphi_{\mu}(\theta)=G_{\mu}({\textstyle \frac{\vert\theta\vert}{\pi}})=-\sqrt{{\textstyle \frac{\cot^{2}(\theta)}{\eta^{2}+\cot^{2}(\theta)}}}1_{\left\{ \vert\theta\vert\in(0,\frac{\pi}{2})\right\} }+\sqrt{{\textstyle \frac{\cot^{2}(\theta)}{\eta^{2}+\cot^{2}(\theta)}}}1_{\left\{ \vert\theta\vert\in(\frac{\pi}{2},\pi)\right\} }.
\]
Using Mathematica software (many thanks to Professor Rajai Nacer for
his assistance), we obtained the following $\mu$-domain, or equivalently
the Brownian symmetrization of $U$.

\begin{figure}[H]
\begin{centering}
\includegraphics[width=6cm,height=6cm,keepaspectratio]{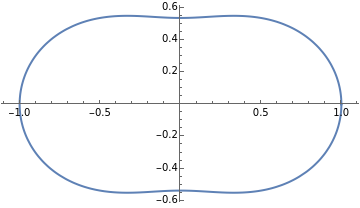}
\par\end{centering}
\caption{$\mathfrak{B}(U)$ is limited by the blue curve.}

\end{figure}

Using the area formula \eqref{eqn:area}, we found that the area of
$\mathfrak{B}(U)$ is approximately $0.6$. In particular 
\[
\mathfrak{B}(U)\subsetneq\mathfrak{S}(U)=\mathbb{D}.
\]
In light of this example, we think the following conjecture might
be true :
\begin{enumerate}
\item Brownian symmetrization increases the principal Dirichlet eigenvalue. 
\item Brownian symmetrization does not increase the area and the perimeter. 
\end{enumerate}
Moreover, we speculate that Brownian symmetrization does better than
Steiner symmetrization. 

\bibliographystyle{plain}
\bibliography{Browniansymmetrization}

\end{document}